\documentclass[11 pt]{amsart}
\usepackage{amsmath,amsthm,amssymb,enumerate, hyperref}
\usepackage{natbib}

\newcommand{\weak}{\stackrel{w}{\longrightarrow}}
\newcommand{\prob}{\stackrel{P}{\longrightarrow}}
\newcommand{\eid}{\stackrel{d}{=}}
\newcommand{\one}{{\bf 1}}

\newcommand{\reals}{{\mathbb R}}
\newcommand{\bbr}{\reals}

\newcommand{\bbn}{\protect{\mathbb N}}
\newcommand{\bbc}{\protect{\mathbb C}}

\newcommand{\A}{{\mathcal A}}
\newcommand{\F}{{\mathcal F}}

\newtheorem{theorem}{Theorem}[section]

\newtheorem{fact}{Fact}[section]

\newtheorem{remark}{Remark}[section]

\def\Var{{\rm Var}}
\def\E{{\rm E}}
\def\eesd{{\rm EESD}}
\def\esd{{\rm ESD}}

\DeclareMathOperator{\Tr}{Tr} 
\DeclareMathOperator{\rank}{Rank}

\numberwithin{equation}{section}

\begin{document}


\title[Folklore of free independence]{A note on the folklore of free independence}
\author{Arijit Chakrabarty}
\address{Theoretical Statistics and Mathematics Unit \\
Indian Statistical Institute \\
203 B. T. Road\\
Kolkata 700108, India}
\email{arijit.isi@gmail.com}

\author{Sukrit Chakraborty}
\address{Theoretical Statistics and Mathematics Unit \\
Indian Statistical Institute \\
203 B. T. Road\\
Kolkata 700108, India}
\email{sukrit049@gmail.com}

\author{Rajat Subhra Hazra}
\address{Theoretical Statistics and Mathematics Unit \\
Indian Statistical Institute \\
203 B. T. Road\\
Kolkata 700108, India}
\email{rajatmaths@gmail.com}


\begin{abstract}
It is shown that a Wishart matrix of standard complex normal random variables is asymptotically freely independent of an independent random matrix, under minimal conditions, in two different sense of asymptotic free independence. 
\end{abstract}

\maketitle

\section{Introduction}\label{sec:intro}
Since the seminal discovery of \cite{voiculescu:1991}, there have been several folklores regarding free independence. For example, one such folklore is that any two independent Wigner matrices are asymptotically freely independent, and another is that any Wishart matrix is asymptotically freely independent of a deterministic matrix. While such folklores are true, more often than not, there are a few problems. The first and foremost problem is that the meaning of the phrase ``asymptotically freely independent'' varies with context. A widely used definition is in terms of the normalized expected trace (or without the expectation). Unfortunately, in this definition, the claim of asymptotic free independence can easily fail, in the absence of any other assumption. The counter example in \cite{male:2017} is noteworthy. This articulates the second problem with the folklore, which is that the assumptions are usually missing. Nevertheless, in the literature, there are several rigorous proofs of various versions of Voiculescu's theorem, see for example, the monographs \cite{nica:speicher:2006}, \cite{anderson:guionnet:zeitouni:2010} and \cite{mingo:speicher:2017}. The reader will notice that the versions in the above references are not monotonic in strength, that is, one version does not necessarily imply another. In other words, there is no general theorem regarding asymptotic free independence from which most results of interest follow.

This note is a modest attempt at settling some of the issues mentioned above in a specific example. Theorems \ref{t1} and \ref{t.main} claim asymptotic free independence of a Wishart matrix $W_N$ of standard complex normal random variables and an independent matrix $Y_N$, in two different definitions of asymptotic free independence. The former is the usual definition, in terms of normalized expected trace, while the latter is in terms of the limiting spectral distribution of random matrices, which is weaker than the former. In both the above theorems, the limiting spectral distribution of $Y_N$ is assumed to be compactly supported, at the  least. This assumption is relaxed in Theorem \ref{t3}, a consequence of which is that the claim is also significantly weakened. The proofs of Theorems \ref{t.main} and \ref{t3} is based on a truncation argument.

The authors chose to work with the complex normal distribution because they yield the strongest results in that the assumptions on $Y_N$ become minimal. Theorem 22.35 of \cite{nica:speicher:2006} astutely notes this, for example. It's worth noting that the results in \cite{hiai:petz:2000} are similar in spirit. Although the results are stated for a Wishart matrix, they hold for a Wigner matrix too.

\section{The results}\label{sec:results}
Let $(Z_{i,j}:i,j\in\bbn)$ be a family of i.i.d.\ standard complex Normal random variables. That is, $(\Re(Z_{i,j}):i,j\ge1)$ and $(\Im(Z_{i,j}):i,j\ge1)$ are independent families of i.i.d.\ real random variables from $N(0,1/2)$. Suppose that $(M_N:N\ge1)$ is a sequence of positive integers such that
\begin{equation}
\label{eq:lambda}\lim_{N\to\infty}\frac{N}{M_N}=\lambda\in(0,\infty)\,.
\end{equation}
For each $N\ge1$, let $X_N$ be the $M_N\times N$ random matrix defined by
\[
X_N(i,j):=Z_{i,j},\,1\le i\le M_N,\,1\le j\le N\,.
\]
For $N\ge1$, define an $N\times N$ random Hermitian matrix by
\[
W_N:=\frac1{M_N}X_N^*X_N\,.
\]
Notice that for $1\le i,j\le N$,
\[
W_N(i,j)=\frac1{M_N}\sum_{k=1}^{M_N}\overline{Z_{k,i}}Z_{k,j}\,.
\]
Hence, $W_N$ is a Wishart matrix. 

For a random Hermitian $N\times N$ matrix $Z$, its ``empirical spectral distribution'' and ``expected empirical spectral distribution'', denoted by $\esd(Z)$ and $\eesd(Z)$, respectively, are probability measures on $\bbr$, defined as
\begin{eqnarray*}
\esd(Z)&=&\frac1N\sum_{i=1}^N\one(\lambda_i\in\cdot)\,,\\
\eesd(Z)&=&\frac1N\sum_{i=1}^NP(\lambda_i\in\cdot)\,,
\end{eqnarray*}
where $\lambda_1,\ldots,\lambda_N$ are the eigenvalues of $Z$, counted with multiplicity.

It is well known that as $N\to\infty$,
\[
\esd(W_N)\to\nu_\lambda\,,
\]
weakly in probability, where $\nu_\lambda$, with $\lambda$ as in \eqref{eq:lambda}, is the Mar\v{c}enko-Pastur distribution, defined by
\[
\nu_\lambda(dx)=
\begin{cases}
\left(1-\frac1\lambda\right)\one(0\in dx)+\frac1{2\pi}\frac{\sqrt{(\lambda_+-x)(x-\lambda_-)}}{\lambda x}\one_{[\lambda_-,\lambda_+]}(x)\,dx,&\lambda>1\,,\\
\frac1{2\pi}\frac{\sqrt{(\lambda_+-x)(x-\lambda_-)}}{\lambda x}\one_{[\lambda_-,\lambda_+]}(x)\,dx,&\lambda\le1\,,
\end{cases}
\]
with $\lambda_\pm=(1\pm\sqrt\lambda)^2$.

For each $N\ge1$, $Y_N$ is an $N\times N$ random complex Hermitian matrix, \textbf{independent} of $(Z_{i,j}:i,j\in\bbn)$. The exact assumption on the spectrum of $Y_N$ will vary from result to result, and hence will be mentioned in the statements of the respective results. However, at the very least, there exists a (non-random) probability measure $\mu$ on $\bbr$ such that
\begin{equation}
\label{t.main.assume}\esd(Y_N)\to\mu\,,
\end{equation}
weakly in probability, as $N\to\infty$.

The statements of the following results are based on the theory of $C^*$-probability spaces. A reader unacquainted with this may look at \cite{nica:speicher:2006}. It is known that given probability measures $\mu_1$ and $\mu_2$ which are supported on a compact subset of $\bbr$, there exists a $C^*$-probability space $(\A,\varphi)$ and two freely independent self-adjoint elements $a_1,a_2\in\A$ such that
\[
\varphi\left(a_i^n\right)=\int_{-\infty}^\infty x^n\mu_i(dx),\,n\in\bbn,\,i=1,2\,.
\]
The probability measures $\mu_1$ and $\mu_2$ are called the distributions of $a_1$ and $a_2$, and denoted by ${\mathcal L}(a_1)$ and ${\mathcal L}(a_2)$, respectively.

The first result shows asymptotic free independence between $W_N$ and $Y_N$ in the sense of normalized expected trace.

\begin{theorem}\label{t1}
Assume that  $\mu$ is compactly supported, and that for each $n\in\bbn$,
\begin{eqnarray}
\label{f2.eq1}\lim_{N\to\infty}\E\left[\frac1N\Tr(Y_N^n)\right]&=&\int_{-\infty}^\infty x^n\mu(dx)\,,\\
\label{f2.eq2}\text{ and }\lim_{N\to\infty}\Var\left[\frac1N\Tr(Y_N^n)\right]&=&0\,.
\end{eqnarray}
Then, there exists a $C^*$-probability space $(\A,\varphi)$, in which there are two freely independent self-adjoint elements $w$ and $y$, having distribution $\nu_\lambda$ and $\mu$, respectively, and satisfying the following. For every polynomial $p$ in two variables having complex coefficients,
\begin{equation}
\label{t1.claim}\lim_{N\to\infty}\frac1N\E\Tr\left[p\left(W_N,Y_N\right)\right]=\varphi\bigl(p(w,y)\bigr)\,.
\end{equation}
Consequently, if $p\left(W_N,Y_N\right)$ has real eigenvalues, a.s., for all $N$, then as $N\to\infty$,
\begin{equation}
\label{t.main.claim}\eesd\left(p\left(W_N,Y_N\right)\right)\weak{\mathcal L}\left(p\left(w,y\right)\right)\,.
\end{equation}
\end{theorem}

\begin{remark}
When $Y_N$ is deterministic, the assumptions of Theorem \ref{t1} just mean that
\[
\lim_{n\to\infty}\frac1N\E\left(\Tr(Y_N^n)\right)=\int_{-\infty}^\infty x^n\mu(dx)\,,
\]
which is stronger than \eqref{t.main.assume}.
\end{remark}

\begin{remark}\label{rem1}
The claim \eqref{t.main.claim} is an immediate consequence of \eqref{t1.claim}, whenever $p$ is such that the eigenvalues of $p(W_N,Y_N)$ are a.s.\ real. For example, $W_N$ is non-negative definite implies that the above holds for
\[
p(x,y)=xy\,.
\]
\end{remark}

In the next result, both -- the hypotheses and the claim are weakened to \eqref{t.main.assume} and \eqref{t.main.claim}, respectively. In other words, this results proves asymptotic free independence in the sense of \eqref{t.main.claim} as opposed to \eqref{t1.claim}.

\begin{theorem}\label{t.main}
If $\mu$, as in \eqref{t.main.assume} is compactly supported, then, for every polynomial $p$ in two variables having complex coefficients,  such that $p\left(W_N,Y_N\right)$ has real eigenvalues, a.s., for all $N$, then \eqref{t.main.claim} holds.
\end{theorem}

The last result deals with the case when the support of  $\mu$ is possibly unbounded. For measures with possibly unbounded support, `$\boxplus$' and `$\boxtimes$' denote their free additive and multiplicative convolutions, respectively. For the latter, at least of one of the two measures has to be supported on the non-negative half line. See \cite{bercovici:voiculescu:1993} for the details.

\begin{theorem}\label{t3}
If \eqref{t.main.assume} holds for a probability measure $\mu$ which is not necessarily compactly supported, then 
\begin{eqnarray*}
\eesd(Y_N+W_N)&\weak&\mu\boxplus\nu_\lambda\,,\\
\text{and, }\eesd(Y_NW_N)&\weak&\mu\boxtimes\nu_\lambda\,,
\end{eqnarray*}
as $N\to\infty$.
\end{theorem}

\begin{remark}\label{rem.wigner}

Theorems \ref{t1} -- \ref{t3} hold true, if Wishart matrix is replaced by a Wigner matrix with standard complex normal entries, and Mar\v{c}enko-Pastur distribution is replaced by the semicircle law.

\end{remark}

\section{Some facts}\label{sec:facts}

For the proofs of the results mentioned in Section \ref{sec:results}, a few facts will be needed, which are stated here. The proofs are omitted because the results are either elementary or can be found in a cited reference.

The first one is a comparison between ranks of deterministic matrices.

\begin{fact}\label{f1}
Let $p$ be a polynomial in two variables, with complex coefficients. Then, there exists a finite constant $C$, depending only on the polynomial $p$, such that
\[
\rank\left(p(A,B)-p(A^\prime,B)\right)\le C\rank(A-A^\prime)\,,
\]
for square matrices $A,A^\prime,B$ of the same order. 
\end{fact}

The next result, which is also based on rank, follows  from Theorem A.43, page 503 of \cite{bai:silverstein:2010}.

\begin{fact}\label{f7}
For probability measures $\mu_1$ and $\mu_2$ on $\bbr$, let $d(\mu_1,\mu_2)$ denote their sup distance, defined by
\[
d(\mu_1,\mu_2):=\sup_{x\in\bbr}\left|\mu_1\bigl((-\infty,x]\bigr)-\mu_2\bigl((-\infty,x]\bigr)\right|\,.
\]
For $N\times N$ random Hermitian matrices $A$ and $B$, it holds that
\[
d\left(\eesd(A),\eesd(B)\right)\le\frac1N\E\left[\rank(A-B)\right]\,.
\]
\end{fact}

The next two facts are elementary.

\begin{fact}\label{f2}
For each $N\ge1$, suppose that $Y_N$ is an $N\times N$ random Hermitian matrix satisfying \eqref{f2.eq1} and \eqref{f2.eq2}. Then, it holds that for any $n\ge1$ and $k_1,\ldots,k_n\ge0$,
\[
\lim_{N\to\infty}N^{-n}\E\left(\prod_{i=1}^n\Tr\left(Y_N^{k_i}\right)\right)=\prod_{i=1}^n\alpha_{k_i}\,,
\]
where $\alpha_n$ denotes the right hand side of \eqref{f2.eq1}, for $n\ge1$.
\end{fact}

\begin{fact}\label{f3}
Let $Z_1,\ldots,Z_N$ follow i.i.d.\ standard complex Normal, that is, for each $i=1,\ldots,N$, real and imaginary parts of $Z_i$ are independent $N(0,1/\sqrt2)$. If $Z$ denotes the column vector whose $i$-th component is $Z_i$, and $U$ is an $N\times N$ deterministic unitary matrix, then the components of $UZ$ are also i.i.d.\ from standard complex Normal. 
\end{fact}

The next fact has essentially been proved in page 386 of \cite{nica:speicher:2006}. As mentioned therein, an $N\times N$ Haar unitary matrix is a random matrix distributed according to the Haar measure on the group of $N\times N$ unitary matrices. Before stating the fact, we need to introduce a few notations. Let $S_n$ denote the group of permutations on $\{1,\ldots,n\}$, for $n\ge1$. A permutation is identified with the partition of $\{1,\ldots,n\}$, induced by the cyclic decomposition. For $\alpha\in S_n$, $\#\alpha$ denotes the number of blocks in $\alpha$, that is, the number of cycles. For any block $\theta\in\alpha$, $\#\theta$ denotes the length of the cycle $\theta$. For example, for
\[
\alpha\in S_4
\] defined by
\[
\alpha(1)=2,\,\alpha(2)=4,\,\alpha(3)=3,\,\alpha(4)=1\,,
\]
we write
\[
\alpha=\{(1,2,4),(3)\}\,,
\]
and hence $\#\alpha=2$. If the elements of $\alpha$, as listed above, are labelled as $\theta_1$ and $\theta_2$, respectively, then
\[
\#\theta_1=3,\,\#\theta_2=1\,.
\]

\begin{fact}\label{f4}
For a fixed $N$, let $A$ and $B$ be deterministic $N\times N$ Hermitian matrices. If $U$ is an $N\times N$ Haar unitary matrix, then for any $1\le n\le N$ and $k_1,\ldots,k_n\ge0$,
\begin{eqnarray*}
&&\E\Tr\left[\prod_{i=1}^n\left(UA^{k_i}U^*B\right)\right]\\
&=&\sum_{\alpha,\beta\in S_n}Wg(N,\alpha^{-1}\beta)\left(\prod_{\theta\in\alpha}\Tr\left(A^{\sum_{i\in\theta}k_i}\right)\right)\left(\prod_{\theta\in\beta^{-1}\gamma}\Tr\left(B^{\#\theta}\right)\right)\,,
\end{eqnarray*}
where $Wg$ is the Weingarten function defined by
\[
Wg(N,\alpha)=\E\left[U(1,1)\ldots U(n,n)\overline{U(1,\alpha(1))}\ldots\overline{U(n,\alpha(n))}\right]\,,
\]
for $\alpha\in S_n$, $N\ge n$ and
\[
\gamma=\{(1,\ldots,n)\}\in S_n\,.
\]
\end{fact}

The following has essentially been proved in the course of the proof of Theorem 23.14 of \cite{nica:speicher:2006}.

\begin{fact}\label{f5}
For a fixed $n\ge1$ and $\alpha\in S_n$,
\[
\lim_{N\to\infty}N^{2n-\#\alpha}Wg(N,\alpha)=\phi(\alpha)\in\bbr\,.
\]
Furthermore, if $(\A,\varphi)$, $w$ and $y$ are as in the statement of Theorem \ref{t1}, then for $n\ge1$ and $k_1,\ldots,k_n\ge0$,
\begin{eqnarray*}
&&\varphi\left(w^{k_1}y\ldots w^{k_n}y\right)\\
&=&\sum_{\alpha,\beta\in S_n:\#(\alpha^{-1}\beta)+\#\alpha+\#(\beta^{-1}\gamma)=2n+1}\Biggl[\phi(\alpha^{-1}\beta)\left(\prod_{\theta\in\alpha}\varphi\left(w^{\sum_{i\in\theta}k_i}\right)\right)\\
&&\,\,\,\,\,\,\,\,\,\,\,\,\,\,\,\,\left(\prod_{\theta\in\beta^{-1}\gamma}\varphi\left(y^{\#\theta}\right)\right)\Biggr]\,.
\end{eqnarray*}
\end{fact}

The following result is Corollary 2 of \cite{azoff:1974}.

\begin{fact}\label{f6}
For a fixed $N\in\bbn$, there exists a measurable map
\[
\psi:\bbc^{N\times N}\to\bbc^{N\times N}\,,
\]
where $\bbc^{N\times N}$ is the space of all $N\times N$ matrices with complex entries, such that $\psi(M)$ is an unitary matrix for every $M\in\bbc^{N\times N}$, and
\[
\psi(M)^*M\psi(M)
\]
is upper triangular for every $M$.
\end{fact}

\section{Proofs}\label{sec:proofs}

\begin{proof}[Proof of Theorem \ref{t1}]
Let $(\A,\varphi)$, $w$ and $y$ be as in the statement. In order to prove the claim, all that needs to be shown is that
\begin{equation}
\label{t1.eq3}\lim_{N\to\infty}N^{-1}\E\left[\Tr\left(W_N^{k_1}Y_N\ldots W_N^{k_n}Y_N\right)\right]=\varphi\left(w^{k_1}y\ldots w^{k_n}y\right)\,,
\end{equation}
for fixed $n\ge1$ and $k_1,\ldots,k_n\ge0$. 

The foremost task is to show that the expectation on the left hand side of \eqref{t1.eq3} exists. To that end, it suffices to show that there exists $N_0$ such that
\begin{equation}
\label{t1.eq7}\E\left[|Y_N(i,j)|^n\right]<\infty\text{ for all }N\ge N_0,\,1\le i,j\le N\,.
\end{equation}
Since
\begin{eqnarray*}
\sum_{i,j=1}^N\E\left[|Y_N(i,j)|^{2n}\right]
&\le&\E\left[\Tr(Y_N^{2n})\right]\,,
\end{eqnarray*}
and \eqref{f2.eq1} implies that right hand side is finite for $N$ large,  an $N_0$ satisfying \eqref{t1.eq7} exists.

Proceeding towards \eqref{t1.eq3}, fix $N\ge N_0$ and let 
\[
\F:=\sigma\left(X_N,Y_N\right)\,,
\]
that is $\F$ is the smallest $\sigma$-field with respect to which the entries of $X_N$ and $Y_N$ are measurable. Let $U_N$ be a Haar unitary matrix, independent of $\F$. Fact \ref{f3} implies that conditioned on $U_N$, the entries of $U_NX_N$ are i.i.d.\ standard complex Normal. That is, the conditional joint distribution of the entries of $U_NX_N$, given $U_N$, is same as that of $X_N$. Therefore,
\[
\left(U_NW_NU_N^*,Y_N\right)\eid\left(W_N,Y_N\right)\,.
\]
As a result,
\begin{eqnarray*}
C_N&:=&\E\left[\Tr\left(W_N^{k_1}Y_N\ldots W_N^{k_n}Y_N\right)\right]\\
&=&\E\left[\Tr\left((U_NW_NU_N^*)^{k_1}Y_N\ldots(U_NW_NU_N^*)^{k_n}Y_N\right)\right]\\
&=&\E\left[\Tr\left(U_NW_N^{k_1}U_N^*Y_N\ldots U_NW_N^{k_n}U_N^*Y_N\right)\right]\\
&=&\E\,\E_\F\left[\Tr\left(U_NW_N^{k_1}U_N^*Y_N\ldots U_NW_N^{k_n}U_N^*Y_N\right)\right]\,,
\end{eqnarray*}
where $\E_\F$ is the conditional expectation given $\F$. By an appeal to Fact \ref{f4},
\begin{eqnarray*}
&&\E_\F\left[\Tr\left(U_NW_N^{k_1}U_N^*Y_N\ldots U_NW_N^{k_n}U_N^*Y_N\right)\right]\\
&=&\sum_{\alpha,\beta\in S_n}Wg(N,\alpha^{-1}\beta)\left(\prod_{\theta\in\alpha}\Tr\left(W_N^{\sum_{i\in\theta}k_i}\right)\right)\left(\prod_{\theta\in\beta^{-1}\gamma}\Tr\left(Y_N^{\#\theta}\right)\right)\,.
\end{eqnarray*} 
Taking the unconditional expectation of both sides, and using the independence of $W_N$ and $Y_N$, we get that
\begin{equation}\label{t1.eq2}
C_N=\sum_{\alpha,\beta\in S_n}Wg(N,\alpha^{-1}\beta)\E\left(\prod_{\theta\in\alpha}\Tr\left(W_N^{\sum_{i\in\theta}k_i}\right)\right)\E\left(\prod_{\theta\in\beta^{-1}\gamma}\Tr\left(Y_N^{\#\theta}\right)\right)\,.
\end{equation}

It is well known that for all $k\in\bbn$,
\begin{eqnarray*}
\lim_{N\to\infty}\E\left(N^{-1}\Tr(W_N^k)\right)&=&\varphi(w^k)\,,\\
\lim_{N\to\infty}\Var\left(N^{-1}\Tr(W_N^k)\right)&=&0\,.
\end{eqnarray*}
Combining the above with Fact \ref{f2} yields that
\begin{equation}
\label{t1.eq4}\lim_{N\to\infty}\E\left(\prod_{\theta\in\alpha}N^{-1}\Tr\left(W_N^{\sum_{i\in\theta}k_i}\right)\right)=\prod_{\theta\in\alpha}\varphi\left(w^{\sum_{i\in\theta}k_i}\right)\,.
\end{equation}
Similarly, \eqref{f2.eq1}, \eqref{f2.eq2} and Fact \ref{f2} together imply that
\begin{equation}
\label{t1.eq5}\lim_{N\to\infty}\E\left(\prod_{\theta\in\beta^{-1}\gamma}N^{-1}\Tr\left(Y_N^{\#\theta}\right)\right)=\prod_{\theta\in\beta^{-1}\gamma}\varphi\left(y^{\#\theta}\right)\,.
\end{equation}
Rewrite \eqref{t1.eq2} as
\begin{eqnarray*}
&&N^{-1}C_N\\
&=&\sum_{\alpha,\beta\in S_n}N^{\#\alpha+\#(\beta^{-1}\gamma)-1}Wg(N,\alpha^{-1}\beta)\\
&&\,\,\,\,\,\,\,\,\,\E\left(\prod_{\theta\in\alpha}N^{-1}\Tr\left(W_N^{\sum_{i\in\theta}k_i}\right)\right)\E\left(\prod_{\theta\in\beta^{-1}\gamma}N^{-1}\Tr\left(Y_N^{\#\theta}\right)\right)\,.
\end{eqnarray*}
The first claim of Fact \ref{f5} implies that for fixed $\alpha,\beta\in S_n$,
\begin{eqnarray}
\nonumber N^{\#\alpha+\#(\beta^{-1}\gamma)-1}Wg(N,\alpha^{-1}\beta)
\label{t1.eq6}&=&O\left(N^{\#(\alpha^{-1}\beta)+\#\alpha+\#(\beta^{-1}\gamma)-2n-1}\right)\\
\nonumber&=&O(1)\,,
\end{eqnarray}
because
\begin{eqnarray*}
\#\alpha+\#(\alpha^{-1}\beta)+\#(\beta^{-1}\gamma)\le2n+1\,,
\end{eqnarray*}
as shown in (23.4) and the following display on page 387 in \cite{nica:speicher:2006}. Therefore, letting $N\to\infty$ in \eqref{t1.eq6} and using the first claim of Fact \ref{f5} along with \eqref{t1.eq4} and \eqref{t1.eq5}, we get that
\begin{eqnarray*}
&&\lim_{N\to\infty}N^{-1}C_N\\
&=&\sum_{\alpha,\beta\in S_n:\#(\alpha^{-1}\beta)+\#\alpha+\#(\beta^{-1}\gamma)=2n+1}\Biggl[\phi(\alpha^{-1}\beta)\left(\prod_{\theta\in\alpha}\varphi\left(w^{\sum_{i\in\theta}k_i}\right)\right)\\
&&\,\,\,\,\,\,\,\,\,\,\,\,\,\,\,\,\left(\prod_{\theta\in\beta^{-1}\gamma}\varphi\left(y^{\#\theta}\right)\right)\Biggr]\,.
\end{eqnarray*}
The second claim of Fact \ref{f5} shows that the right hand side of the above equation is same as that of \eqref{t1.eq3}.
Thus, the latter follows, which completes the proof.
\end{proof}

\begin{proof}[Proof of Theorem \ref{t.main}]
Since $\mu$ is compactly supported, let $M>1$ be such that
\[
\mu\left([-(M-1),M-1]\right)=1\,.
\]

Letting $\psi$ be as in Fact \ref{f6}, define
\[
P_N=\psi(Y_N)\,,
\]
and
\[
T_N:=P_N^*Y_NP_N\,,
\]
which is an upper triangular matrix. Define an $N\times N$ matrix $T^\prime_N$ by
\[
T^\prime_N(i,j):=
\begin{cases}
T_N(i,j),&i\neq j\,,\\
T_N(i,i)\one(|T_N(i,i)|\le M),&i=j\,,
\end{cases}
\]
and let
\begin{equation}\label{eq.defynp}
Y^\prime_N:=P_NT_N^\prime P_N^*\,.
\end{equation}
In order to complete the proof, it suffices to show that for a fixed polynomial $p$ satisfying the hypothesis, 
\begin{equation}
\label{t.main.eq1}\eesd\left(p(W_N,Y_N^\prime)\right)\weak{\mathcal L}\left(p(w,y)\right)\,,
\end{equation}
and
\begin{equation}
\label{t.main.eq2}\lim_{N\to\infty}d\left(\eesd\left(p(W_N,Y_N)\right),\eesd\left(p(W_N,Y_N^\prime)\right)\right)=0\,.
\end{equation}

We start with showing \eqref{t.main.eq2}. To that end, note that
\begin{eqnarray*}
N^{-1}\rank(Y_N-Y^\prime_N)&=&N^{-1}\rank(T_N-T_N^\prime)\\
&\le&N^{-1}\#\{1\le i\le N:|T_N(i,i)|>M\}\\
&=&\left(\esd(Y_N)\right)\left([-M,M]^c\right)\,,
\end{eqnarray*}
the inequality in the second line being based on the fact that $T_N-T_N^\prime$ is a diagonal matrix, and hence
\begin{equation}
\label{t.main.eq3}N^{-1}\rank(Y_N-Y^\prime_N)\prob0\,,
\end{equation}
as $N\to\infty$. Fact \ref{f1} and the bounded convergence theorem show that 
\[
\lim_{N\to\infty}\E\left[\frac1N\rank(p(W_N,Y_N)-p(W_N,Y_N^\prime))\right]=0\,.
\]
An appeal to Fact \ref{f7} establishes \eqref{t.main.eq2}.

Proceeding towards \eqref{t.main.eq1}, in view of Theorem \ref{t1} and Remark \ref{rem1}, it suffices to show that \eqref{f2.eq1} and \eqref{f2.eq2} hold with $Y_N$ replaced by $Y_N^\prime$. Equation \eqref{t.main.eq3} and the hypotheses imply that
\[
\esd(Y_N^\prime)\to\mu\,,
\]
weakly in probability, as $N\to\infty$. Since 
\[
\left(\esd(Y_N^\prime)\right)\left([-M,M]^c\right)=\left(\esd(T_N^\prime)\right)\left([-M,M]^c\right)=0,\,N\ge1\,,
\]
and
\[
\mu\left([-M+1,M-1]^c\right)=0\,,
\]
it follows that for a fixed $n\ge1$, as $N\to\infty$,
\[
\int_{-\infty}^\infty x^n\left(\esd(Y_N^\prime)\right)(dx)\prob\int_{-\infty}^\infty x^n\mu(dx)\,.
\]
The observations that 
\[
\frac1N\Tr\left[(Y_N^\prime)^n\right]=\int_{-\infty}^\infty x^n\left(\esd(Y_N^\prime)\right)(dx)\,,
\]
and that the modulus of the above quantity is bounded by $M$, show that \eqref{f2.eq1} and \eqref{f2.eq2} hold, with $Y_N$ replaced by $Y_N^\prime$. Theorem \ref{t1} shows \eqref{t.main.eq1}, which in turn completes the proof.
\end{proof}

\begin{proof}[Proof of Theorem \ref{t3}]
As in the preceding proof, let 
\[
P_N=\psi(Y_N),\,N\ge1\,.
\]
Fix $M>0$ and let $Y_N^\prime$ be as in \eqref{eq.defynp}, $M$ being suppressed in the notation. Theorem \ref{t.main} implies that 
\[
\eesd(Y_N^\prime+W_N)\weak\mu_M\boxplus\nu_\lambda\,,
\]
and
\[
\eesd(Y_N^\prime W_N)\weak\mu_M\boxtimes\nu_\lambda\,,
\]
as $N\to\infty$, where 
\[
\mu_M(B)=\mu(B\cap[-M,M])+\mu\left([-M,M]^c\right)\one(0\in B)\,,
\]
for every Borel set $B\subset\bbr$. Proposition 4.13 and Corollary 6.7 of \cite{bercovici:voiculescu:1993} imply, respectively, that as $M\to\infty$,
\begin{eqnarray*}
\mu_M\boxplus\nu_\lambda&\weak&\mu\boxplus\nu_\lambda\,,\\
\text{and }\mu_M\boxtimes\nu_\lambda&\weak&\mu\boxtimes\nu_\lambda\,.
\end{eqnarray*}
In order to complete the proof, in view of Fact \ref{f1}, it suffices to show that
\begin{equation*}
\lim_{M\to\infty}\limsup_{N\to\infty}\frac1N\E\left[\rank(Y_N-Y_N^\prime)\right]=0\,.
\end{equation*}
However, arguments as in the proof of Theorem \ref{t.main} show that for $M$ such that
\[
\mu(\{-M,M\})=0\,,
\]
it holds that
\[
\limsup_{N\to\infty}\frac1N\E\left[\rank(Y_N-Y_N^\prime)\right]\le\mu\left([-M,M]^c\right)\,.
\]
Hence, the proof follows.
\end{proof}


\end{document}